\documentclass{article}
\pdfoutput=1
\usepackage{fullpage}
\usepackage{amsfonts,amssymb,amsmath}
\usepackage{bm, bbm}
\usepackage{authblk}
\usepackage[round]{natbib}

\usepackage[dvipsnames]{xcolor}
\usepackage{url}
\usepackage{graphicx}

\usepackage{amsthm, thmtools, thm-restate}

\expandafter\def\csname ver@etex.sty\endcsname{3000/12/31}

\usepackage{autonum}

\usepackage{sectsty}
\usepackage{caption}
\usepackage{subcaption}
\usepackage[normalem]{ulem}
\sectionfont{\sffamily\bfseries\upshape\large}
\subsectionfont{\sffamily\bfseries\upshape\normalsize}
\subsubsectionfont{\sffamily\bfseries\upshape\normalsize}
\makeatletter
\renewcommand\@seccntformat[1]{\csname the#1\endcsname.\quad}
\newtheorem{theorem}{Theorem}
\newtheorem*{theorem*}{Theorem}
\newtheorem{corollary}[theorem]{Corollary}

\newtheorem{lemma}[theorem]{Lemma}
\newtheorem*{lemma*}{Lemma}

\newtheorem*{definition*}{Definition}

\newtheorem*{fact*}{Fact}

\DeclareMathOperator*{\argmin}{\arg\!\min}

\newcommand\E{\mathbb{E}}

\newcommand\diag{\operatorname{diag}}
\newcommand\var{\operatorname{var}}
\newcommand\R{\mathbb{R}}
\newcommand\X{\mathcal{X}}
\newcommand\Y{\mathcal{Y}}
\newcommand\T{{\scriptscriptstyle{\mathsf{T}}}}
\newcommand\corr{\operatorname{corr}}

\newtheorem{observe}{Observation}[section]
\newtheorem{remark1}[observe]{Remark}

\def\1{\mathbbm{1}}

\title{\vspace{-.6in}The piranha problem: Large effects swimming in a small pond\footnote{To appear in the {\em Notices of the American Mathematical Society}.  We thank Lauren Kennedy, the editor, and four reviewers for helpful discussion. CT and DH acknowledge NSF grant CCF-1740833, a JP Morgan Faculty Award, and a Sloan Research Fellowship. Much of this work was done while CT was at Columbia University. AG acknowledges ONR grant N000142212648. BG acknowledges NSF grants 2051246 and 2153019.  AV acknowledges Research Council of Finland project 340721.
Affiliations:
Memorial Sloan Kettering Cancer Center (CT), Department of Statistics, Columbia University (PG and AG), Department of Political Science, Columbia University (BG and AG), Department of Computer Science, Aalto University (AV), Department of Computer Science, Columbia University (DH).}\vspace{-.1in}}

\author{Christopher $\!$Tosh, Philip $\!$Greengard, Ben $\!$Goodrich, Andrew $\!$Gelman, Aki $\!$Vehtari, Daniel $\!$Hsu\vspace{-.1in}}

\date{2 Apr 2024}

\begin{document}

\maketitle

\vspace{-30pt}
\section{Background}\label{background}
\vspace{-3pt}

In this work, we discuss an inevitable consequence of having a stable system in which many explanatory variables have large effects: these variables must have large interactions which will be unlikely to cancel either other out to the extent required for general stability or predictability.
We call this type of result a ``piranha theorem'' \citep{gelman2017}, the analogy being the folk belief that if one has a large number of piranhas (representing large effects) in a single fish tank, then one will soon be left with far fewer piranhas.
If there is some outcome for which studies find large and consistent effects of many different inputs, then we can conclude that some of these effects are smaller than claimed or that multiple explanatory variables are essentially measuring the same phenomenon.

Identifying and measuring the effects of explanatory variables are central problems in statistics and drive much of the world's scientific research.
Despite the substantial effort spent on these tasks, there has been comparatively little work on addressing a related question: how many explanatory variables can have large effects on an outcome?  The present work follows up on \citet{cornfield:1959} and \citet{ding:2014}, considering quantitative constraints in the effects of additional variables.

Consider, for example, the problem of explaining voters' behaviors and choices.
Researchers have identified and tested the effects of internal factors such as fear, hope, pride, anger, anxiety, depression, and menstrual cycles, as well as external factors such as droughts, shark attacks, and the performance of local college football teams. Many of these findings have been questioned on methodological grounds, but they remain in the public discourse.  Beyond the details of these particular studies, it is natural to ask if all of these effects can be real in the sense of representing patterns that will consistently appear in the future.

The implication of the published and well-publicized claims regarding ovulation and voting, shark attacks and voting, college football and voting, etc., is not merely that some voters are superficial and fickle. No, this literature claims that seemingly trivial or irrelevant factors have {\em large and consistent} effects, and this runs into the problem of interactions. For example, the effect on your vote of the local college football team losing could depend crucially on whether there's been a shark attack lately, or on what's up with your hormones on election day. Or the effect could be positive in an election with a female candidate and negative in an election with a male candidate. Or the effect could interact with your parents' socioeconomic status, or whether your child is a boy or a girl, or the latest campaign ad, or any of the many other factors that have been studied in the evolutionary psychology and political psychology literatures.  If such effects are large, it is necessary to consider their interactions, because the average effect of a factor in any particular study will depend on the levels of all the other factors in that environment. Similarly, \citet{mellon2020rain} has argued against naive assumptions of causal identification in economics, where there is a large literature considering rainfall as an instrumental variable, without accounting for the implication that these many hypothesized causal pathways would, if taken seriously, represent violations of the assumptions of the model.  Even if a particular experiment or observational study analyzes only one causal factor, the existence of potential interactions (as indeed are implied if one were to take the social science literature at face value) destroys the implicit assumption that an effect measured under some particular set of conditions can be interpreted as a general or persistent effect.

These concerns are particularly relevant in social science, where the search for potential causes is open-ended.
Our work here is partly motivated by the replication crisis, which refers to the difficulties that many have had in trying to independently verify established findings in social and biological sciences.
Many of the explanations for the crisis have focused on various methodological issues, such as low power and unrecognized researcher degrees of freedom~\citep{SNS11}.
Beyond the criticisms of practice and suggested fixes, these works have also provided much needed statistical intuition.
Groups of studies that claim to have found a variety of important explanatory variables for a single outcome should be scrutinized, particularly when the dependencies among the explanatory variables have not been investigated. 

This article collects several mathematical results regarding the distributions of correlations or coefficients, with the aim of fostering further work on statistical models for environments with a multiplicity of effects.  What is novel in this paper is not the theorems themselves but rather viewing them in the context of trying to make sense of clusters of research studies that claim to have found large effects.

There are many ways to capture the dependence among random variables, and thus we should expect there to be a correspondingly large collection of piranha theorems.
We formalize and prove piranha theorems for correlation, regression, and mutual information in Section~\ref{sec:theorems}.
These theorems illustrate the general phenomena at work in any setting with multiple causal or explanatory variables, with implications for the replication crisis in social science.

\vspace{-6pt}
\section{Piranhas and butterflies}
\vspace{-3pt}

A fundamental tenet of social psychology and behavioral economics, at least how it is presented in the news media, and taught and practiced in many business schools, is that small ``nudges,'' often the sorts of things that we might not think would affect us at all, can have big effects on behavior.

The model of the world underlying these claims is not just the ``butterfly effect'' that small changes can have big effects; rather, it's that small changes can have big and predictable effects, a sort of ``button-pushing'' model of social science, the idea that if you do $A$, you can expect to see $B$.

In response to this attitude, we present the piranha argument, which states that there can be some large and predictable effects on behavior, but not a lot, because, if there were, then these different effects would interfere with each other, a ``hall of mirrors'' of interactions \citep{cronbach:1975} that would make it hard to detect any consistent effects of anything in observational data.

In a similar vein, \cite{cook2018} writes:

\vspace{-8pt}
\begin{quotation}\noindent
``The butterfly effect is the semi-serious claim that a butterfly flapping its wings can cause a tornado half way around the world. It's a poetic way of saying that some systems show sensitive dependence on initial conditions, that the slightest change now can make an enormous difference later \dots Once you think about these things for a while, you start to see nonlinearity and potential butterfly effects everywhere. There are tipping points everywhere waiting to be tipped!''
\end{quotation}
\vspace{-8pt}
\noindent%
But, Cook continues, it's not so simple:
\vspace{-7pt}
\begin{quotation}\noindent
``A butterfly flapping its wings usually has no effect, even in sensitive or chaotic systems. You might even say especially in sensitive or chaotic systems. Sensitive systems are not always and everywhere sensitive to everything. They are sensitive in particular ways under particular circumstances and can otherwise be resistant to influence.\,\dots The lesson that many people draw from their first exposure to complex systems is that there are high-leverage points, if only you can find them and manipulate them. They want to insert a butterfly at just the right time and place to bring about a desired outcome. Instead, we should humbly evaluate to what extent it is possible to steer complex systems at all. We should evaluate what aspects can be steered and how well they can be steered. The most effective intervention may not come from tweaking the inputs but from changing the structure of the system.''
\end{quotation}
\vspace{-6pt}

Effects in social science vary across people and scenarios and over time, and they can be represented by probability distributions.  Cook's advice to think about ``the structure of the system'' echoes recommendations from the literature on statistical quality control that system-level variation puts a limit on what can be learned about the average effects of particular interventions.  In the presence of possible interactions, there is no reason to expect stability of treatment effects.

\vspace{-6pt}
\section{Example:  hypothesized effect sizes in social priming}\label{priming}
\vspace{-3pt}

We demonstrate the possibility of quantitative analysis of the piranha problem using the example of an influential experiment from 1996 in which participants were given a scrambled-sentence task and then were surreptitiously timed when walking away from the lab.
Students whose sentences included elderly-related words such as ``worried,'' ``Florida,'' ``old,'' and ``lonely'' walked an average of 13\% more slowly than students in the control condition, and the difference was statistically significant.

This experimental claim is of historical interest in psychology in that, despite its implausibility, it was taken seriously for many years and received thousands of citations, but it failed to replicate and is no longer generally believed to represent a real effect; for background see \cite{wagenmakers2015}.  Now we understand such apparently statistically-significant findings as the result of selection with many researcher degrees of freedom \citep{SNS11}.
Here, though, we will take the published claim at face value and also work within its larger theoretical structure, under which weak indirect stimuli can produce large effects.

An effect of 13\% on walking speed is not in itself huge; the difficulty comes when considering elderly-related words as just one of many potential stimuli.  Here are just some of the factors that have been published in the social priming and related literatures as having large effects on attitudes and behavior:  hormones, subliminal images, news of football games and shark attacks, a chance encounter with a stranger, parental socioeconomic status, weather, the last digit of one's age, the sex of a hurricane name, the sexes of siblings, the position in which a person is sitting, and many others.  See \cite{controversies} for references to these claims, along with other papers criticizing or refuting them.  A common feature of these examples is that the stimuli have no clear direct effect on the measured outcomes, and in many cases the experimental subject is not even aware of the manipulation.  Based on these examples, one can come up with dozens of other potential stimuli that fit the pattern.  In addition to walking speed being affected by elderly-related words, one could also consider word lengths (with longer words corresponding to slower movement), sounds of words (with smooth sibilance motivating faster walking), subject matter (sports-related words as compared to sedentary words), affect (happy compared to sad words, or calm compared to angry), words related to travel (inducing faster walking) or invoking adhesives such as tape or glue (inducing slower walking), and so on.  Similarly, one can consider different sorts of incidental events, not just encounters with strangers but also a ringing phone or knocking at the door, the presence of a male or female lab assistant (which could have a main effect or interact with the participant's sex), a newspaper or magazine on a nearby table, \emph{ad infinitum}.

Now we can invoke the piranha principle.  Imagine 100 possible stimuli, each with an effect of 13\% on walking speed, all of which could arise in a real-world setting where we encounter many sources of text, news, and internal and external stimuli.  If each stimulus corresponds to two equally probable states with effects of $\pm 0.5\log(1.13)$ on log walking speed, and these effects are independent in the wild, then the sum of these will be approximately normally distributed with standard deviation $0.5\sqrt{100}\log(1.13)=0.61$, thus walking speed could easily be multiplied or divided by $e^{0.61} = 1.8$ based on a collection of arbitrary stimuli that are imperceptible to the person being affected.  And this factor of 1.8 could be made arbitrarily large by simply increasing the number of potential primes.

It is outrageous to think that walking speed could be doubled or halved based on a random collection of unnoticed and essentially irrelevant stimuli---but that is the implication of the embodied cognition literature.  It is basically a Brownian motion model in which the individual inputs are too large to work out.

We can think of five ways to avoid the ridiculous conclusion.   The first possibility is that the different factors could interact or interfere in some way so that the variance of the total effect is less than the sum of the variances of the individual components, or multiple explanatory variables could be essentially measuring the same phenomenon.  Second, effects could be much smaller.  Change those 13\% effects to 1\% effects and you can get to more plausible totals, in the same way that real-world Brownian oscillations are tolerable because the impact of each individual molecule in the liquid is so small.  Third, one could reduce the total number of possible influences.  If there were only 10 possible stimuli rather than 100 or 1000 or more, then the total effect could fall within the range of plausibility.  Fourth, there could be a distribution of effects with a few large influences and a long tail of relatively unimportant factors, so that, when correctly translated to standardized population effect sizes, most treatment effects are already small, and the infinite sum has a reasonable bound. Fifth, multiple explanatory variables could be essentially measuring the same phenomenon.

All these options have major implications for the study of social priming and, more generally, for causal inference in an open-ended setting with large numbers of potential influences.  If large interactions are possible, this suggests that stable individual treatment effects might be impossible to find:  a 13\% effect of a particular intervention in one particular experiment might be $-18\%$ in another context or $+2\%$ in the presence of some other unnoticed factor, and this in turn raises questions about the relevance of any particular study.  If effects are much smaller than reported, this suggests that existing studies are extremely underpowered, so that published estimates are drastically overestimated and often in the wrong direction \citep{gelman2014}, thus essentially noise.  At the same time, a restriction of the universe of potential stimuli would require an overhaul of the underlying theoretical framework in which just about any stimulus can cause a noticeable change.  For example, if we think there cannot be more than five or ten large effects on walking speed, it would seem a stretch that unnoticed words in a sentence scrambling test would be one of these special factors.  Finally, if the distribution of average effects is represented by a long series, most of whose elements are tiny (either because of small individual effects or because any large effects occur infrequently in the general population, as with rare diseases or short-term interventions with rapidly decaying effects), this implies a prior distribution on average effect sizes with a spike near zero, which in turn would result in near-zero estimated population effect sizes in most cases.  Our point is not that all effects are zero but rather that in a world of essentially infinitely many possible causal factors, some external structure must be applied in order to identify stable effects  from finite samples.

\vspace{-6pt}
\section{Piranha theorems}
\vspace{-3pt}
\label{sec:theorems}

In this section, we present piranha theorems for linear and nonlinear effects.
We consider two different ways of measuring linear effects.
We first show that it is impossible for a large number of explanatory variables to be correlated with some outcome variable unless they are highly correlated with each other. 
Second, we show that if a set of explanatory random variables is plugged into a linear regression, the $\ell_2$-norm $\|\beta\|$ of the least-squares coefficient vector $\beta$ can be bounded above in terms of the eigenvalues of the second-moment matrix of the predictors.
Thus, there can only be so many individual coefficients with a large magnitude.
Finally, we consider a general nonlinear form of dependency, mutual information, and present a corresponding piranha theorem for that measure.

\vspace{-6pt}
\subsection{Correlation}
\vspace{-3pt}

The first type of pattern we consider is correlation.
In particular, we will show that if all the covariates are highly correlated with some outcome variable, then there must be a reasonable amount of correlation among the covariates themselves.
This is formalized in the following theorem, which is known as Van der Corput's inequality \citep{T14}.
We offer a proof here for completeness.
\vspace{-3pt}

\begin{theorem}[Van der Corput's inequality]
  \label{thm: correlation piranha theorem}
  If $X_1, \dotsc, X_p, y$ are real-valued random variables with finite nonzero variance, then
\vspace{-9pt}
  \[ \sum_{i=1}^p |\corr(X_i,y)| \ \leq \ \sqrt{p + \sum_{i \neq j} |\corr(X_i,X_j)|}.  \]
\vspace{-9pt}

\noindent%
In particular, if $|\corr(X_i,y)| \geq \tau$ for each $i=1,\dotsc, p$, then $\sum_{i \neq j} |\corr(X_i,X_j)| \geq p(\tau^2 p - 1)$.
\end{theorem}
\vspace{-9pt}
\begin{proof}
  Without loss of generality, we may assume that $X_1,\dotsc,X_p,y$ have mean zero and unit variance.
  Define $Z_1,\dotsc,Z_p$ by
  \vspace{-9pt}
  \[ 
    Z_i \ = \
    \begin{cases}
      \hphantom{-}X_i & \text{if $\E(yX_i) > 0$,} \\
      -X_i & \text{else.}
    \end{cases}
  \]
\vspace{-6pt}

\noindent%
Thus $\E(yZ_i) = |\E(yX_i)|$ and $\E(Z_i^2) = \E(X_i^2)$ for each $i=1,\dotsc,p$.
  By Cauchy-Schwarz,
  \vspace{-6pt}
  \[
    \sum_{i=1}^p \E( y Z_i)
    \ = \  \E\left( y  \sum_{i=1}^p  Z_i \right)
    \ \leq \ \sqrt{ \E\left( \left( \sum_{i=1}^p Z_i \right)^2 \right) } .
  \]
\vspace{-18pt}

\noindent%
Therefore,
  \vspace{-9pt}
  \[
    \sum_{i=1}^p |\E(y X_i )|
    \ = \  \sum_{i=1}^p \E(y Z_i )
    \ \leq \ \sqrt{ \sum_{i=1}^p \E(Z_i^2) + \sum_{i \neq j} \E(Z_i Z_j) }
    \ \leq \ \sqrt{ p + \sum_{i \neq j} |\E(X_i X_j)| } .
  \]
  \vspace{-12pt}
  
\noindent%
Rearranging gives us the theorem statement.
\end{proof}
\vspace{-6pt}

A direct consequence of Theorem~\ref{thm: correlation piranha theorem} is that if $X_1, \dotsc, X_p$ are independent (or uncorrelated) random variables and each has correlation at least $\tau$ with $y$, then $\tau \leq 1/\sqrt{p}$.

In some situations, the outcome may change from study to study, for example a program evaluation in economics might look at employment, income, or savings; a political intervention might target turnout or vote choice; or an education experiment might look at several tests.  Although the different outcomes in a study are not exactly the same, we might reasonably expect them to be highly correlated. However, if we have mean-zero and unit-variance random variables $x,y,z$ satisfying $\E(xy) \geq \tau$ and $\E(yz) \geq 1 - \epsilon$, then

\vspace{-9pt}
\[ \E(xz) \ = \ \E(x (z-y + y)) \ \geq \ \tau + \E(x(z - y)) , \]
\vspace{-15pt}

\noindent%
and, by Cauchy-Schwarz,
\vspace{-6pt}
\[ \E(x(z - y))^2 \ \leq \ \E(x^2) \E((z- y)^2) \ \leq \ 2 - 2(1-\epsilon). \]
Thus, $\E(xz) \geq \tau - \sqrt{2\epsilon}$.
This gives the following corollary of Theorem~\ref{thm: correlation piranha theorem}.

\begin{corollary}
\label{cor: multiple correlation piranha corollary}
Suppose $X_1, Y_1, \dotsc, X_p, Y_p$ are real-valued random variables with finite nonzero variance. If $\corr(Y_i,Y_j) \geq 1 - \epsilon$ and $|\corr(X_i,Y_i)| \geq \tau$ for $i,j=1,\dotsc,p$, then $\sum_{i \neq j} |\corr(X_i,X_j)| \geq p((\tau-\sqrt{2\epsilon})^2 p - 1)$.
\end{corollary}

The bound in Theorem~\ref{thm: correlation piranha theorem} is essentially tight for large $p$. To see this, pick any $0 \leq \tau \leq 1$ and take $X_1, \dotsc, X_p$ to be mean-zero random variables with covariance matrix $\Sigma$ given by
\vspace{-9pt}
\[
  \Sigma_{ij} \ = \ 
  \begin{cases} 
    1 & \text{ if } i = j , \\
    \tau^2 & \text{ if } i \neq j .
  \end{cases}
\]
\vspace{-15pt}

\noindent
If $y = \sum_{j=1}^p X_j$, then for each $i=1,\dotsc,p$,
\vspace{-6pt}
\[
  \corr(X_i,y)
  \ = \ \frac{\E\left( X_i \sum_{j=1}^p X_j \right)}{\sqrt{\E\left(\sum_{j,k} X_j X_k\right)}}
  \ = \ \frac{1+(p-1)\tau^2}{\sqrt{p + p(p-1)\tau^2}}
  \ \stackrel{p\to\infty}{\longrightarrow} \tau .
\]
\vspace{-9pt}

One drawback of Theorem~\ref{thm: correlation piranha theorem} is that the upper bound depends on a coarse measure of interdependence of the covariates, namely the sum of all pairwise correlations $\sum_{i,j} |\corr(X_i, X_j)|$. One might hope that if we have a finer-grained control on the correlation matrix, we should be able to get a stronger result. This is accomplished by the following piranha theorem, which shows that we can instead get an upper bound that depends on the largest eigenvalue of the correlation matrix. However, this comes at the expense of bounding the sum of squared correlations $|\corr(X_i, Y)|^2$, rather than the sum of their absolute values.

\begin{theorem}
\label{thm: correlation eigenvalue piranha theorem}
If $X_1, \dotsc, X_p, y$ are real-valued random variables with finite nonzero variance, then
\vspace{-6pt}
  \[ \sum_{i=1}^p |\corr(X_i,y)|^2 \ \leq \ \lambda_{\max},  \]
\vspace{-6pt}

\noindent%
where $\lambda_{\max}$ is the maximum eigenvalue of the correlation matrix $\Sigma_{i,j} = \corr(X_i,X_j)$.
\end{theorem}

Consider again the case where $X_1, \ldots, X_p$ are uncorrelated, and each has correlation at least $\tau$ with $y$. In this case, the correlation matrix will be the identity matrix, whose largest eigenvalue is 1, and Theorem~\ref{thm: correlation eigenvalue piranha theorem} implies that $\tau \leq 1/\sqrt{p}$, which was the same conclusion provided by Theorem~\ref{thm: correlation piranha theorem}. However, in general Theorems~\ref{thm: correlation piranha theorem} and \ref{thm: correlation eigenvalue piranha theorem} are incomparable since $\sum_{i=1}^p |\corr(X_i,y)|^2 \leq \sum_{i=1}^p |\corr(X_i,y)|$, but 
\[ \lambda_{\max}^2 \leq \sum_{i,j} |\corr(X_i,X_j)|^2  \leq  \sum_{i,j} |\corr(X_i,X_j)| . \vspace{-4pt} \]
As an example of when these theorems can produce different conclusions, one can give a randomized construction of a correlation matrix $\Sigma = A^T A$, where the columns of $A$ are drawn from the uniform distribution over the hypersphere $S^{p-1}$. In this case, if each covariate has correlation at least $\tau$ with $y$, then with high probability the conclusion of Theorem~\ref{thm: correlation piranha theorem} is that $\tau$ is bounded above on the order of $1/\sqrt[4]{p}$, while Theorem~\ref{thm: correlation eigenvalue piranha theorem} gives a much tighter bound on the order of $1/\sqrt{p}$.

The proof of Theorem~\ref{thm: correlation eigenvalue piranha theorem} relies on the following technical lemma, essentially a consequence of orthogonality.
\vspace{-15pt}

\begin{lemma}
\label{lem: squared covariance bound}
If $U_1, \dotsc, U_p, y$ are real-valued random variables with mean zero and unit variance such that $\E(U_i U_j) = 0$ for all $i \neq j$, then $\sum_{i=1}^p  \left(\E U_i y \right)^2 \ \leq \ 1.$
\end{lemma}

\begin{proof}
Denote the covariance matrix of the random vector $(U_1, \dotsc, U_p, y)^\T$ as
\vspace{-6pt}
\[
  \Sigma
  \ = \ 
  \begin{pmatrix}
    I  & a \\
    a^\T & 1
  \end{pmatrix} ,
\]
\vspace{-9pt}

\noindent%
where $a_i = \E \left( U_i y \right)$ for $i=1,\dotsc,p$.
Define the vector $v = (-a^\T,\|a\|)^\T \in \R^{p+1}$.
Then
\vspace{-6pt}
\[ v^\T \Sigma v \ = \ 2 (1-\|a\|) \|a\|^2 \ \geq \ 0 , \]
\vspace{-18pt}

\noindent%
where the inequality follows from the fact that $\Sigma$ is a covariance matrix and hence positive semi-definite.
We conclude that $\| a \| \leq 1$.
\vspace{-6pt}
\end{proof}

With the above in hand, we turn to the proof of Theorem~\ref{thm: correlation eigenvalue piranha theorem}.

\begin{proof}[Proof of Theorem~\ref{thm: correlation eigenvalue piranha theorem}]
Assume without loss of generality that $X_1, \ldots, X_p, y$ have mean zero and unit variance.
Denote the eigendecomposition of $\Sigma$ as $Q \diag(\lambda_1,\dotsc,\lambda_p) Q^T$, where $\lambda_1 \geq \dotsb \geq \lambda_p \geq 0$ and $Q$ is orthogonal.

Let $\tilde{X} = Q^T X$, where $X = (X_1, \ldots, X_p)$. Then $\tilde{X} = (\tilde{X}_1,\dotsc,\tilde{X}_p)$ is a mean-zero random vector whose covariance matrix is $\diag(\lambda_1,\dotsc,\lambda_p)$.
For $j \in \{1,\dotsc,p\}$ with $\lambda_j = \var(\tilde{X}_j) = 0$, we have $\tilde{X}_j = 0$ almost surely.
We then apply Lemma~\ref{lem: squared covariance bound} to get
\vspace{-6pt}
\[ \|\E(y \tilde X)\|^2 \ = \ \sum_{j=1}^p \E(y\tilde X_j)^2 \ = \ \sum_{j:\lambda_j > 0} \lambda_j \E(y \tilde X_j / \sqrt{\lambda_j})^2 \ \leq \ \lambda_1 \sum_{j:\lambda_j > 0} \E(y \tilde X_j / \sqrt{\lambda_j})^2 \ \leq \ \lambda_1 . \]
\vspace{-12pt}

\noindent%
Then,
\vspace{-9pt}
\[ \sum_{i=1}^p |\corr(X_i,y)|^2 \ = \ \| \E(y {X}) \|^2 \ = \ \| QQ^T \E(y {X}) \|^2 \ = \ \| Q \E(y \tilde{X}) \|^2 \ = \ \|\E(y \tilde{X}) \|^2 \ \leq \  \lambda_1 , \]
\vspace{-6pt}

\noindent%
where we have used the fact that $Q$ is orthogonal.
\end{proof}

\vspace{-6pt}
\subsection{Linear regression}

We next turn to showing that least squares linear regression solutions cannot have too many large coefficients. Specifically, letting $\beta = (\beta_1,\dotsc,\beta_p)^\T \in \R^p$ denote the regression coefficients of least squared error,
\vspace{-6pt}
\begin{align}
\label{eqn: regression}
{\beta} \ = \ \argmin_{\alpha = (\alpha_1,\dotsc,\alpha_p)^\T \in \R^p} \E\left(\left(\alpha_1 X_1 + \cdots + \alpha_p X_p  - y\right)^2 \right),
\end{align}
\vspace{-6pt}

\noindent%
we bound the number of $\beta_i$'s that can have large magnitude.
This is formalized in our next piranha theorem.

\begin{theorem}
\label{thm: regression piranha theorem}
Suppose $X_1, \dotsc, X_p, y$ are real-valued random variables with mean zero and unit variance. If $\beta \in \R^p$ satisfies equation~\eqref{eqn: regression}, then the squared $\ell_2$ norm of $\beta$ satisfies
\vspace{-3pt}
\[ \| \beta \|^2 \ \leq \ \frac{1}{\lambda_{\min}} , \]
\vspace{-12pt}

\noindent%
where $\lambda_{\min}$ is the minimum eigenvalue of the second-moment matrix $\E(XX^\T)$ of $X = (X_1,\dotsc,X_p)^\T$.
\end{theorem}

Consider again the setting where $X_1, \dotsc, X_p$ are standardized and uncorrelated. In this case, the second-moment matrix $\E(XX^\T)$ will be the identity matrix, and its minimum eigenvalue $\lambda_{\min}$ will be $1$.
Thus, Theorem~\ref{thm: regression piranha theorem} states for independent covariates, there may be at most $1/\kappa^2$ regression coefficients $\beta_i$ with magnitude larger than $\kappa$. In general, $\lambda_{\min}$ cannot get small without the explanatory variables having sizeable correlations with each other.

\begin{proof}[Proof of Theorem~\ref{thm: regression piranha theorem}]
  The case where $\lambda_{\min} = 0$ is trivial.
  Thus, assume $\lambda_{\min} > 0$.
  In this case, the second-moment matrix $\E(XX^\T)$ is invertible, its inverse has eigenvalues bounded above by $1/\lambda_{\min}$, and
  \begin{align}
    \beta & \ = \ (\E(XX^\T))^{-1} \E(yX) .
  \end{align}
  Define $\tilde X = (\E(XX^\T))^{-1/2} X$, so $\tilde X = (\tilde X_1,\dotsc,\tilde X_p)^\T$ is a vector of mean-zero and unit-variance random variables with $\E(\tilde X_i\tilde X_j) = 0$ for all $i \neq j$.
  By Lemma~\ref{lem: squared covariance bound},
  \vspace{-6pt}
  \[
    \|\E(y\tilde X)\|^2 \ = \ \sum_{j=1}^p \E(y\tilde X_j)^2 \ \leq \ 1 .
  \]
  \vspace{-18pt}

\noindent%
Therefore,
\vspace{-6pt}
  \[
    \|\beta\|^2
    \ = \ \|(\E(XX^\T))^{-1/2} \E(y\tilde X)\|^2
    \ = \ \E(y \tilde X)^\T (\E(XX^\T))^{-1} \E(y \tilde X)
    \ \leq \ \frac1{\lambda_{\min}} \|\E(y \tilde X)\|^2
    \ \leq \ \frac1{\lambda_{\min}} ,
  \]
  where the first inequality uses the upper bound of $1/\lambda_{\min}$ on the eigenvalues of $(\E(XX^\T))^{-1}$.
\end{proof}

\vspace{-6pt}
\subsection{Mutual information}
\label{sec:information}
\vspace{-3pt}

Though many statistical analyses hinge on discovering linear relations among variables, not all do.
Thus, we turn to a more general form of dependency for random variables, mutual information.
Our mutual information piranha theorem will be of a similar form as the previous results, namely that if many covariates share information with a common variable, then they must share information among themselves.

To simplify our analysis, we assume that all the random variables we consider in this section take values in discrete spaces.
For two random variables $x$ and $y$, their mutual information is defined as
\vspace{-6pt}
\[ I(x ; y) \ = \ H(x) - H(x \, | \, y) \ = \ H(y) - H(y \, | \, x), \]
\vspace{-18pt}

\noindent%
where $H(\cdot)$ and $H(\cdot \, | \, \cdot)$ denote entropy and conditional entropy, respectively. These are defined as
\vspace{-6pt}
\begin{align}
H(x) \ &= \ \sum_{x \in \X} p(x) \log \frac{1}{p(x)} , \\
H(y \, | \, x) \ &= \ \sum_{x \in \X, y \in \Y} p(x, y) \log \frac{p(x)}{p(x, y)} ,
\end{align} 
\vspace{-9pt}

\noindent%
where $\X$ (resp.\ $\Y$) is the range of $x$ (resp.\ $y$), $p(x,y)$ is the joint probability mass function of $x$ and $y$, and $p(x)$ is the marginal probability mass function of $x$.

We use the following facts about entropy and conditional entropy. 
\vspace{-6pt}
\begin{fact*}[Chain rule of entropy]
For random variables $X_1, \dotsc, X_p$,
\vspace{-6pt}
\[0  \ \leq \ H(X_1, \dotsc, X_p) \ = \ \sum_{i=1}^p H(X_i \, | \, X_1, \dotsc, X_{i-1}) .\]
\vspace{-15pt}

\noindent%
Moreover, for any other random variable $y$,
\vspace{-9pt}
\[0  \ \leq \ H(X_1, \dotsc, X_p \, | \, y) \ = \ \sum_{i=1}^p H(X_i \, | \, y, X_1, \dotsc, X_{i-1}) .\]
\end{fact*}
\vspace{-15pt}

\begin{fact*}[Conditioning reduces entropy]
For random variables $x, y, z$,
\vspace{-6pt}
\[ H(x | y, z) \ \leq \ H(x \, | \, y) \ \leq \ H(x). \]
\end{fact*}
\vspace{-6pt}

Using these facts, we can prove the following piranha theorem about mutual information.

\begin{theorem}
\label{thm: almost independent mutual info piranha theorem}
Given random variables $X_1, \dotsc, X_p$ and $y$, we have
\vspace{-6pt}
\[ \sum_{i=1}^p I(X_i ; y) \ \leq \ H(y) + \sum_{i=1}^p I(X_i ; X_{-i}), \]
\vspace{-9pt}

\noindent%
where $X_{-i} = (X_1, \dotsc, X_{i-1}, X_{i+1}, \dotsc, X_p)$.
\end{theorem}
\vspace{-6pt}
\begin{proof}
Using the definition of mutual information, we have
\vspace{-4pt}
\[ H(X_{i} \, | \, X_{-i}) \ \geq H(X_i) - I(X_i ; X_{-i}). \]
\vspace{-17pt}

\noindent%
Since conditioning reduces entropy, this implies
\vspace{-4pt}
\[ H(X_i \, | \, X_1, \dotsc, X_{i-1}) \ \geq \ H(X_{i} \, | \, X_{-i}) \ = \  H(X_i) -  I(X_i ; X_{-i}). \]
\vspace{-17pt}

\noindent%
Then, by the chain rule of entropy,
\vspace{-9pt}
\begin{equation}
\label{eqn: entropy lower bound}
 H(X_1, \dotsc, X_p) \ = \ \sum_{i=1}^p H(X_i \, | \, X_1, \dotsc, X_{i-1}) \ \geq \ \sum_{i=1}^p H(X_i) -  I(X_i ; X_{-i}) .
\end{equation}
\vspace{-9pt}

\noindent%
The chain rule of entropy combined with the fact that conditioning reduces entropy implies
\vspace{-6pt}
\begin{equation}
\label{eqn: conditional entropy upper bound}
 H(X_1, \dotsc, X_p \, | \, y) \ \leq \  \sum_{i=1}^p H(X_i \, | \, y).
\end{equation}
\vspace{-12pt}

\noindent%
Plugging equations~\eqref{eqn: entropy lower bound} and~\eqref{eqn: conditional entropy upper bound} into our formula for $I(X_1, \dotsc, X_p ; y)$ gives
\vspace{-6pt}
\begin{align}
 I(X_1, \dotsc, X_p ; y) \ &= \ H(X_1, \dotsc, X_p) - H(X_1, \dotsc, X_p \, | \, y)  \\ 
\ &\geq \ \sum_{i=1}^p  H(X_i) - I(X_i ; X_{-i}) - H(X_i \, | \, y) \\
 \ &= \ \sum_{i=1}^p I(X_i ; y) - I(X_i ; X_{-i}).
\end{align}
\vspace{-15pt}

\noindent%
Now we can also write,
\vspace{-6pt}
\[  I(X_1, \dotsc, X_p ; y)  \ = \ H(y) - H(y \, | \, X_1, \dotsc, X_p) \ \leq \ H(y). \]
\vspace{-18pt}

\noindent%
Rearranging yields the theorem.
\end{proof}
\vspace{-3pt}

\noindent%
One corollary of Theorem~\ref{thm: almost independent mutual info piranha theorem} is that for any random variable $y$, there can be at most $p \leq H(y)/\gamma$ random variables $X_1,\dotsc, X_p$ that (a) are mutually independent and (b) satisfy $I(X_i ; y) \geq \gamma$.

\vspace{-6pt}
\section{Correlations in a finite sample}
\vspace{-3pt}
\label{sec:finite}

We now turn our focus back to correlations, this time in a finite sample. Suppose we conduct a survey with data on $p$ predictors $X$ and one outcome of interest $y$ on a random sample of $n$ people, and then we evaluate the correlations between the outcome and each of the predictors. 

We collect the data in an $n \times p $ matrix $X$ with $n>p$, where each of the columns $X_1,\dotsc,X_p \in \R^n$ of $X$ has mean zero and unit $\ell_2$ norm, and we will use $\corr(x,y)$ for $x, y \in \R^n$ (neither in the span of the all-ones vector $\1$) to denote the sample correlation:
\vspace{-9pt}
\begin{equation}\label{20}
\corr(x, y) \ = \ \frac{\sum_{i=1}^n (x_i - \mu_x) (y_i - \mu_y)}
{\sqrt{\sum_{i=1}^n (x_i-\mu_x)^2\sum_{i=1}^n (y_i-\mu_y)^2}},
\end{equation}
\vspace{-6pt}

\noindent%
where $\mu_x = \tfrac1n \sum_{i=1}^n x_i$ and $\mu_y = \tfrac1n \sum_{i=1}^n y_i$.

An application of Theorem~\ref{thm: correlation eigenvalue piranha theorem} tells us that any non-constant vector $y \in \R^n$ satisfies
\vspace{-6pt}
\[ 0 \ \leq \ \sum_{j=1}^p |\corr(X_i, y)|^2 \ \leq \ \sigma_1^2, \] 
\vspace{-12pt}

\noindent%
where $\sigma_1 \geq \dotsb \geq \sigma_p \geq 0$ denote the singular values of $X$. Moreover, it is not hard to see there exists a vector that achieves the upper bound, namely the top singular vector of $X$.
This analysis shows a \emph{worst-case} piranha theorem: a bound on the number of large correlations with all possible response vectors. Stronger results can be obtained if we consider average behavior. Here, we consider a \emph{stochastic} piranha theorem in which we assume that $y$ is uniformly distributed on the unit sphere in $\R^n$. Our result will hold for any choice of radially symmetric random
vector $y$ that is independent of $X$, but we state it for the uniform distribution over the unit sphere for concreteness. We choose a radially symmetric
distribution because we have no reason to give preference to 
one direction over another. Recall the value of studying average as well as worst-case behavior in areas such as random matrix theory.

The following theorem demonstrates this principle, showing that the maximum sum of squared 
correlations, an $O(1)$ quantity in $n$, is generally 
much larger than the expected sum of square
correlations. 
Specifically, the following theorem shows that the expected sum of squared 
correlations decays like $1/n$.
\begin{theorem} \label{thm:corr-average}
Let $y$ be uniformly distributed on the unit sphere in $\R^n$. Then
\vspace{-9pt}
\begin{equation}
  \E\big(\sum_{i=1}^p \corr(X_i, y)^2\big) \ = \ \frac{p}{n-1} .
\end{equation}
\end{theorem}
\vspace{-9pt}

\noindent
If $y$ is uniformly distributed on the unit sphere
in $\R^n$, then for large $n$, the distribution of $y$ is well approximated 
by $(Z_1,\dotsc,Z_n)$ the $n$-dimensional multivariate Gaussian with mean zero and covariance $\frac{1}{n} I$.
In particular, $(Z_1,\dotsc,Z_n)$ is spherically symmetric, and
\vspace{-6pt}
\begin{align}
   \E(Z_1^2 + \dotsb + Z_n^2) = 1 \qquad \text{and} \qquad \var{(Z_1^2 + \dotsb + Z_n^2)} = O(1/n^2) .
\end{align}
\vspace{-18pt}

\noindent%
As a consequence, for large $n$, the distribution of sum of squared 
correlations is well approximated by a linear combination of 
independent $\chi^2$ random variables, each with one degree of freedom:  $\frac{1}{n - 1} (\lambda_1^2 \xi_1 + \dotsb + \lambda_p^2 \xi_p)$.

Combining this observation with Theorems~\ref{thm: correlation eigenvalue piranha theorem} 
and \ref{thm:corr-average}, for any $n \times p$ matrix (or sample
of data) $X$, if a vector $y$ is distributed according to a spherically
symmetric distribution, then
$\sum_{i=1}^p \corr(X_i, y)^2$ is supported on $[0, \sigma_1^2]$, has expectation $p / (n - 1)$, and for large $n$ has $O(1 / n^2)$ variance.

\vspace{-6pt}
\section{Discussion and directions for future work}
\vspace{-3pt}

The piranha problem is a practical issue: as discussed in the references in Sections \ref{background} and \ref{priming}, it has interfered with research in fields including social priming, evolutionary psychology, economics, and voting behavior. An understanding of the piranha problem can be a helpful step in recognizing fundamental limitations of research in these fields along with related areas of application such as marketing and policy nudges \citep{nonudge:2022}.  We suspect that a naive interpretation of the butterfly effect has led many researchers and policymakers to believe that there can be many large and persistent effects; thus, there is value in exploring the statistical reasons why this is not likely.  In this way, the piranha problem resembles certain other statistical phenomena such as regression to the mean and the birthday coincidence problem, in that there is a regularity in the world that surprises people, and this regularity can be understood as a mathematical result.  This motivates us to seek theorems that capture some of this regularity in a rigorous way. We are not all the way there, but this seems to us to be a valuable research direction.  

\vspace{-6pt}
\subsection{Bridging between deterministic and probabilistic piranha theorems}
\vspace{-3pt}
Are there connections between the worst-case bounds in Section~\ref{sec:theorems}, constraints on main effects and interactions \citep{rogers2002},
the probabilistic bounds in Section~\ref{sec:finite}, 
priors for the effective number of nonzero coefficients,
and models such as the $R^2$ parameterization of linear regression as proposed by \cite{zhang2020r2d2}?  We can consider two directions.  The first is to consider departures from parametric models such as the multivariate normal and $t$ and work out their implications for correlations and regression coefficients.  The second approach is to obtain limiting results in high dimensions (that is, large numbers of predictors), by analogy to central limit theorems of random matrices.  The idea here would be to consider a $n\times (p+1)$ matrix and then pull out one of the columns at random and consider it as the outcome, $y$, with the other $p$ columns being the predictors, $X$.  One should also be able to connect this with work such as \cite{frank2000} and \cite{knaeble2020} on how regression coefficients change when new predictors are added to a model.

\vspace{-6pt}
\subsection{Regularization, sparsity, and Bayesian prior distributions}
\vspace{-3pt}
There has been research from many directions on regularization methods that provide soft constraints on models with large numbers of parameters.  By ``soft constraints,'' we mean that the parameters are not literally constrained to fall within any finite range, but the estimates are pulled toward zero and can only take on large values if the data provide strong evidence in that direction.

Examples of regularization in non-Bayesian statistics include wavelet shrinkage, lasso regression, estimates for overparameterized image analysis and deep learning networks, and models that grow in complexity with increasing sample size.  In a Bayesian context, regularization can be implemented using weakly informative prior distributions \citep[e.g.,][]{greenland2015} or more informative priors that can encode the assumed sparsity \citep[e.g.,][]{Carvalho+etal:2010}.
Classical regularization is motivated by the goal of optimizing long-run frequency performance, and Bayesian priors represent additional information about parameters, coded as probability distributions.  The various piranha theorems correspond to different constraints on these priors and thus even weakly informative priors should start by encoding these constraints.

From a different direction is the idea that any given data might allow only some small number of effects or, more generally, a low-dimensional structure, to be reliably learned.  More generally, models such as the horseshoe \citep{Carvalho+etal:2010} assume a distribution of effect sizes with a sharp peak near zero and a long tail, which represent a solution to the piranha problem by allowing a large number of predictors without overflowing variance.

\vspace{-6pt}
\subsection{Nonlinear models}
\vspace{-3pt}
So far we have discussed linear regression, with theorems capturing different aspects of the constraint that the total $R^2$ cannot exceed 1 (or some bound less than 1, if some of the variation is by its nature unexplainable because it comes from a random process).  We can make similar arguments for nonlinear regression.

For example, consider a model of binary data with 20 causal inputs, each of which is supposed to have an independent effect of 0.5 on the logistic scale.  Aligning these factors in the same direction would give an effect of 10, enough to change the probability from 0.01 to 0.99, which would be unrealistic in applied fields ranging from marketing to voting where no individual behavior can be predicted to that level of accuracy.  One way to avoid these sorts of extreme probabilities would be to suppose the predictors are highly negatively correlated with each other, but in practice, input variables in social science tend to be positively, not negatively correlated (consider, for example, conservative political ideology, Republican party identification, and various issue attitudes that predict Republican vote choice and have positive correlations among the population of voters).  The only other alternative that allows one to keep the large number of large effects is for the model to include strong negative interactions, but then the effects of the individual inputs would no longer be stable, and any effect would depend very strongly on the conditions of the experiment in which it is studied.  It should be possible to express this reasoning more formally, perhaps in a way similar to long-range dependence models in time series and spatial processes.

\vspace{-6pt}
\subsection{Implications for social science research}
\vspace{-3pt}
Although we cannot directly apply these piranha theorems to data, we see them as providing some relevance to social science reasoning.  

As noted at the beginning of this article, there has been a crisis in psychology, economics, and other areas of social science, with prominent findings and apparently strong effects that do not appear in attempted replications by outside research groups;  see, for example, \citet{gordon2020}.  The replication crisis involves many challenges, including estimating its scale and scope, identifying the statistical errors and questionable research practices that have led researchers to systematically overestimate effect sizes and be overconfident in their findings, and studying the incentives of the scientific publication process that can allow entire subfields to get lost in the interpretation of noise.  Even when individual effects are large, they can apply just to a small subset of the population or just for a short period of time, not leaving persistent effects.

The research reviewed in the present article is related to, but different from, the cluster of ideas corresponding to multiple comparisons, false discovery rates, and multilevel models. Those theories correspond to statistical inference in the presence of some specified distribution of effects, possibly very few nonzero effects (the needle-in-a-haystack problem) or possibly an entire continuous distribution, but without necessarily any concern about how these effects interact.  

The present article goes in a different direction, asking the theoretical question:  under what conditions is it possible for many large and persistent effects to coexist in a multivariate system?  In different ways, our results rule out or make extremely unlikely the possibility of multiple large effects or ``piranhas'' among a set of random variables.  These theoretical findings do not directly call into question any particular claimed effect, but they raise suspicions about a model of social interactions in which many large effects are swimming around, just waiting to be captured by researchers who cast out the net of a quantitative study.

To more directly connect our theorems with social science would require modeling predictor and outcome variables in a subfield, similar to multiverse analysis \citep{steegen2016}.  Bounds can be strengthened with reference to empirical correlations among predictors being considered.  When conducting systematic reviews of evidence, it could be appropriate to consider the potential interactions among various hypothesized causal factors, rather than attempting to combine separate estimates using meta-analysis.
Any general implications for social science would only become clear after consideration of particular research areas.

\vspace{-6pt}
\bibliographystyle{abbrvnat}
\bibliography{piranha-references}

\begin{thebibliography}{20}
\providecommand{\natexlab}[1]{#1}
\providecommand{\url}[1]{\texttt{#1}}
\expandafter\ifx\csname urlstyle\endcsname\relax
  \providecommand{\doi}[1]{doi: #1}\else
  \providecommand{\doi}{doi: \begingroup \urlstyle{rm}\Url}\fi

\bibitem[Carvalho et~al.(2010)Carvalho, Polson, and Scott]{Carvalho+etal:2010}
C.~M. Carvalho, N.~G. Polson, and J.~G. Scott.
\newblock The horseshoe estimator for sparse signals.
\newblock \emph{Biometrika}, 97:\penalty0 465--480, 2010.

\bibitem[Cook(2018)]{cook2018}
J.~Cook.
\newblock The other butterfly effect.
\newblock \emph{John D. Cook Consulting},
  \url{https://www.johndcook.com/blog/2018/08/07/the-other-butterfly-effect/},
  2018.

\bibitem[Cornfield et~al.(1959)Cornfield, Haenszel, Hammond, Lilienfeld,
  Shimkin, and Wynder]{cornfield:1959}
J.~Cornfield, W.~Haenszel, E.~C. Hammond, A.~M. Lilienfeld, M.~B. Shimkin, and
  E.~L. Wynder.
\newblock Smoking and lung cancer: Recent evidence and a discussion of some
  questions.
\newblock \emph{Journal of the National Cancer Institute}, 22:\penalty0
  173--203, 1959.

\bibitem[Cronbach(1975)]{cronbach:1975}
L.~J. Cronbach.
\newblock Beyond the two disciplines of scientific psychology.
\newblock \emph{American Psychologist}, 30:\penalty0 116--127, 1975.

\bibitem[Ding and Vanderweele(2014)]{ding:2014}
P.~Ding and T.~J. Vanderweele.
\newblock Generalized {Cornfield} conditions for the risk difference.
\newblock \emph{Biometrika}, 101:\penalty0 971--977, 2014.

\bibitem[Frank(2002)]{frank2000}
K.~A. Frank.
\newblock Impact of a confounding variable on a regression coefficient.
\newblock \emph{Sociological Methods and Research}, 29, 2002.

\bibitem[Gelman(2017)]{gelman2017}
A.~Gelman.
\newblock The piranha problem in social psychology / behavioral economics: The
  ``take a pill'' model of science eats itself, 2017.
\newblock \url{https://statmodeling.stat.columbia.edu/2017/12/15/eats/}.

\bibitem[Gelman(2023)]{controversies}
A.~Gelman.
\newblock Here are just some of the factors that have been published in the
  social priming and related literatures as having large effects on behavior,
  2023.
\newblock \url{https://statmodeling.stat.columbia.edu/2023/11/25/some/}.

\bibitem[Gelman and Carlin(2014)]{gelman2014}
A.~Gelman and J.~B. Carlin.
\newblock Beyond power calculations: Assessing type {S} (sign) and type {M}
  (magnitude) errors.
\newblock \emph{Perspectives on Psychological Science}, 9:\penalty0 641--651,
  2014.

\bibitem[Gordon et~al.(2020)Gordon, Viganola, Bishop, Chen, Dreber, Goldfedder,
  Holzmeister, Johannesson, Liu, Twardy, Wang, and Pfeiffer]{gordon2020}
M.~Gordon, D.~Viganola, M.~Bishop, Y.~Chen, A.~Dreber, B.~Goldfedder,
  F.~Holzmeister, M.~Johannesson, Y.~Liu, C.~Twardy, J.~Wang, and T.~Pfeiffer.
\newblock Are replication rates the same across academic fields? {Community}
  forecasts from the {DARPA SCORE} programme.
\newblock \emph{Royal Society Open Science}, 7:\penalty0 200566, 2020.

\bibitem[Greenland and Mansournia(2015)]{greenland2015}
S.~Greenland and M.~A. Mansournia.
\newblock Penalization, bias reduction, and default priors in logistic and
  related categorical and survival regressions.
\newblock \emph{Statistics in Medicine}, 34:\penalty0 3133--3143, 2015.

\bibitem[Knaeble et~al.(2020)Knaeble, Osting, and Abramson]{knaeble2020}
B.~Knaeble, B.~Osting, and M.~A. Abramson.
\newblock Impact of a confounding variable on a regression coefficient.
\newblock \emph{Epidemiologic Methods}, 9, 2020.

\bibitem[Mellon(2020)]{mellon2020rain}
J.~Mellon.
\newblock Rain, rain, go away: 137 potential exclusion-restriction violations
  for studies using weather as an instrumental variable, 2020.
\newblock \url{https://papers.ssrn.com/sol3/papers.cfm?abstract_id=3715610}.

\bibitem[Rogers(2002)]{rogers2002}
W.~M. Rogers.
\newblock Theoretical and mathematical constraints of interactive regression
  models.
\newblock \emph{Organizational Research Methods}, 5, 2002.

\bibitem[Simmons et~al.(2011)Simmons, Nelson, and Simonsohn]{SNS11}
J.~P. Simmons, L.~D. Nelson, and U.~Simonsohn.
\newblock False-positive psychology: Undisclosed flexibility in data collection
  and analysis allows presenting anything as significant.
\newblock \emph{Psychological Science}, 22:\penalty0 1359--1366, 2011.

\bibitem[Steegen et~al.(2016)Steegen, Tuerlinckx, Gelman, and
  Vanpaemel]{steegen2016}
S.~Steegen, F.~Tuerlinckx, A.~Gelman, and W.~Vanpaemel.
\newblock Increasing transparency through a multiverse analysis.
\newblock \emph{Perspectives on Psychological Science}, 11:\penalty0 702--712,
  2016.

\bibitem[Szászi et~al.(2022)Szászi, Anthony C.~Higney, Gelman, Ziano, Aczel,
  Goldstein, Yeager, and Tipton]{nonudge:2022}
B.~Szászi, A.~B.~C. Anthony C.~Higney, A.~Gelman, I.~Ziano, B.~Aczel, D.~G.
  Goldstein, D.~S. Yeager, and E.~Tipton.
\newblock No reason to expect large and consistent effects of nudge
  interventions.
\newblock \emph{Proceedings of the National Academy of Sciences}, 119:\penalty0
  e2200732119, 2022.

\bibitem[Tao(2014)]{T14}
T.~Tao.
\newblock When is correlation transitive?
\newblock \emph{What's New},
  \url{https://terrytao.wordpress.com/2014/06/05/when-is-correlation-transitive},
  2014.

\bibitem[Wagenmakers et~al.(2015)Wagenmakers, Wetzels, Borsboom, Kievit, and
  van~der Maas]{wagenmakers2015}
E.~Wagenmakers, R.~Wetzels, D.~Borsboom, R.~Kievit, and H.~L.~J. van~der Maas.
\newblock A skeptical eye on psi.
\newblock In \emph{Extrasensory Perception: Support, Skepticism, and Science},
  pages 153--176. Santa Barbara, Calif.: Praeger, 2015.

\bibitem[Zhang et~al.(2020)Zhang, Naughton, Bondell, and Reich]{zhang2020r2d2}
Y.~D. Zhang, B.~P. Naughton, H.~D. Bondell, and B.~J. Reich.
\newblock Bayesian regression using a prior on the model fit: The {R2-D2}
  shrinkage prior.
\newblock \emph{Journal of the American Statistical Association}, 117:\penalty0
  862--874, 2020.

\end{thebibliography}

\appendix

\section{Proof of Theorem~\ref{thm:corr-average}}

For any $x = (x_1,\dotsc,x_n)^\T \in \R^n$ such that $x \neq \lambda \1$ for all $\lambda \in \R$ (i.e., $x$ is not in the span of $\1$), we write $x^* \in \R^n$ to denote the ``standardized'' vector given by the formula,
\vspace{-6pt}
\begin{equation}\label{22}
  x^* \ = \
  \frac{x - \frac1n(x^\T \1) \1}
  {\|x - \frac1n(x^\T \1) \1\|}
  \ = \ \frac{x - (\frac1n \sum_{j=1}^n x_j) \1}{\sqrt{ \sum_{i=1}^n (x_i - \frac1n \sum_{j=1}^n x_j)^2 }}
  .
\end{equation}
\vspace{-9pt}

\noindent%
The unit vector $x^*$  in $\R^n$  is orthogonal to $\1$.
Using this notation, we have
\vspace{-6pt}
\begin{equation} \label{eq:corr-vectors}
  \corr(x,y) \ = \ (x^*)^\T (y^*)
\end{equation}
\vspace{-15pt}

\noindent%
for any $x,y \in \R^n$ not in the span of $\1$.

Write the singular value decomposition of $X$ as
\begin{equation}
\vspace{-9pt}
  \label{eq:svd}
  X \ = \ \sum_{k=1}^p \sigma_k U_k V_k^\T ,
\end{equation}
\vspace{-6pt}

\noindent%
where $U_1,\dotsc,U_p \in \R^n$ are orthonormal left singular vectors of $X$,
$V_1,\dotsc,V_p \in \R^p$ are orthonormal right singular vectors of $X$,
and $\sigma_1 \geq \dotsb \geq \sigma_p \geq 0$ are the singular values of $X$.

Recall that we assume $X_1,\dotsc,X_p$ satisfy $\1^\T X_i = 0$ and $\|X_i\| = 1$ for all $i = 1,\dotsc,p$.
This implies the following lemma.
\vspace{-6pt}
\begin{lemma}
  \label{lem:orthogonal-to-one}
  $X_i = X_i^*$ for all $i=1,\dotsc,p$, and $U_k = U_k^*$ for all $k=1,\dotsc,p$.
\end{lemma}
\begin{proof}
  The assumption on $X_i$ implies that $X_i^* = X_i$ for each $i$.
  Moreover, the assumptions imply that the all-ones vector $\1$ is orthogonal to the range of $X$, which is spanned by $U_1,\dotsc,U_p$.
  Hence $U_k = U_k^*$ for each $k$ as well.
\end{proof}
\vspace{-6pt}

\noindent%
We then take advantage of the following lemma for expressing the sum of squared correlations.

\begin{lemma} \label{lem:corr}
  For any vector $y \in \R^n$ such that $y \neq \lambda \1$ for all $\lambda \in \R$,
  \vspace{-6pt}
  \begin{equation}
    \sum_{i=1}^p \corr(X_i, y)^2
    \ = \ \sum_{k=1}^p \lambda_k^2 (U_k^\T y^*)^2 .
  \end{equation}
\end{lemma}
\vspace{-9pt}

\begin{proof}
  By direct computation:
  \vspace{-6pt}
  \begin{align}
    \sum_{i=1}^p \corr(X_i, y)^2
    & \ = \ \sum_{i=1}^p \left( (X_i^*)^\T (y^*) \right)^2
    && \text{(by equation~\ref{eq:corr-vectors})} \\
    & \ = \ \sum_{i=1}^p \left( X_i^\T y^* \right)^2
    && \text{(by Lemma~\ref{lem:orthogonal-to-one})} \\
    & \ = \ \|X^\T y^*\|^2 \\
    & \ = \ \left\| \sum_{k=1}^p \lambda_k V_k U_k^\T y^* \right\|^2
    && \text{(by equation~\ref{eq:svd})} \\
    & \ = \ \sum_{k=1}^p \lambda_k^2 (U_k^\T y^*)^2
    && \text{(by Pythagorean theorem).}
    \qedhere
  \end{align}
\end{proof}
\vspace{-6pt}

\begin{proof}[Proof of Theorem~\ref{thm:corr-average}]
  By Lemma~\ref{lem:orthogonal-to-one}, the vectors $U_1,\dotsc,U_p$ are orthogonal to the unit vector $\tfrac1{\sqrt{n}}\1$.
  We extend the collection of orthonormal vectors $U_1,\dotsc,U_p,\tfrac1{\sqrt{n}} \1$ with orthonormal unit vectors $U_{p+1},\dotsc,U_{n-1}$ to obtain an orthonormal basis for $\R^n$.
  With probability $1$, the random vector $y$ is not in the span of $\1$.
  Hence, $y^*$ is well defined and can be written uniquely as a linear combination of the aforementioned basis vectors:
  
  \vspace{-6pt}
  \begin{equation}
    y^* \ = \ a_{1}U_{1} + \dotsb + a_{n-1}U_{n-1} + a_n \tfrac1{\sqrt{n}} \1 ,
  \end{equation}
  
  \vspace{-6pt}
\noindent%
  where
  \vspace{-6pt}
  \begin{equation}
    a_k \ = \
    \begin{cases}
      U_k^\T y^* & \text{if $1 \leq k \leq n-1$} , \\
      0 & \text{if $k=n$ (since $\1^\T y^* = 0$)},
    \end{cases}
  \end{equation}
  \vspace{-6pt}

\noindent%
  and
  \vspace{-6pt}
  \begin{equation}
    1 \ = \ a_1^2 + \dotsb + a_{n-1}^2
  \end{equation}
  \vspace{-9pt}

\noindent%
  (since $y^*$ is a unit vector).
  In particular,
$ 1 \ = \ \E( a_1^2 ) + \dotsb + \E( a_{n-1}^2 )$ ,
  which by symmetry implies that $\E( a_k^2 ) \ = \ \frac1{n-1}$
  for each $k=1,\dotsc,n-1$.  Then, by Lemma~\ref{lem:corr},
  \vspace{-6pt}
  \begin{align}
    \E\left( \sum_{i=1}^p \corr(X_i, y)^2 \right)
    & \ = \ \E\left( \sum_{k=1}^p \lambda_k^2 (U_k^\T y^*)^2 \right)
    \ = \ \sum_{k=1}^p \lambda_k^2 \E( a_k^2 )
    \ = \ \frac1{n-1} \sum_{k=1}^p \lambda_k^2
  \end{align}
  \vspace{-6pt}

\noindent%
  Since $\lambda_i^2$ are the eigenvalues of $X^tX$ and the columns of $X$ have unit $\ell_2$ norm,
  \vspace{-6pt}
  \begin{align}
      \ \frac1{n-1} \sum_{k=1}^p \lambda_k^2 = \frac{p}{n-1}
  \end{align}
  \vspace{-9pt}

\noindent%
   because the trace of $X^tX$ is equal to the sum of its eigenvalues.
    \qedhere
\end{proof}

\end{document}